\documentclass[8pt]{article}
\usepackage{indentfirst,latexsym,bm}
\usepackage{times}
\usepackage{amsfonts}
\usepackage{amssymb}
\usepackage[leqno]{amsmath}
\usepackage{dsfont}
\usepackage{amsthm}
\usepackage[all]{xy}
\usepackage{hyperref}
\usepackage{color,xcolor}
\begin{document}
\title{On the realization of a class of $\text{SL}(2,\mathbb{Z})$-representations}
\author{Zhiqiang Yu}
\date{}
\maketitle

\newtheorem{theo}{Theorem}[section]
\newtheorem{prop}[theo]{Proposition}
\newtheorem{lemm}[theo]{Lemma}
\newtheorem{coro}[theo]{Corollary}
\theoremstyle{definition}
\newtheorem{defi}[theo]{Definition}
\newtheorem{exam}[theo]{Example}
\newtheorem{conj}[theo]{Conjecture}
\newtheorem{rema}[theo]{Remark}
\newtheorem{ques}[theo]{Question}

\newcommand{\A }{\mathcal{A}}
\newcommand{\C }{\mathcal{C}}
\newcommand{\B }{\mathcal{B}}
\newcommand{\D }{\mathcal{D}}
\newcommand{\E }{\mathcal{E}}
\newcommand{\FPdim}{\text{FPdim}}
\newcommand{\FQ }{\mathbb{Q}}
\newcommand{\FC }{\mathbb{C}}
\newcommand{\Gal}{\text{Gal}}
\newcommand{\Hom}{\text{Hom}}
\newcommand{\K}{\mathds{k}}
\newcommand{\I }{\mathcal{I}}
\newcommand{\M }{\mathcal{M}}
\newcommand{\Q }{\mathcal{O}}
\newcommand{\rank}{\text{rank}}
\newcommand{\hxs}{\mathfrak{s}}
\newcommand{\hxt}{\mathfrak{t}}
\newcommand{\Rep}{\text{Rep}}
\newcommand{\SL}{\text{SL}}
\newcommand{\ssl}{\mathfrak{sl}}
\newcommand{\vvec}{\text{Vec}}
\newcommand{\W}{\mathcal{W}}
\newcommand{\Y }{\mathcal{Z}}
\newcommand{\Z }{\mathbb{Z}}

\abstract
Let $p<q$ be  odd  primes, $\rho_1$ and $\rho_2$ be irreducible representations of $\SL(2,\Z_p)$ and $\SL(2,\Z_q)$ of dimensions $\frac{p+1}{2}$ and $\frac{q+1}{2}$, respectively. We show that if $\rho_1\oplus\rho_2$ can be realized as modular representation associated to a  modular fusion category $\C$, then $q-p=4$. Moreover, if $\C$ contains a non-trivial \'{e}tale algebra, then  $\C\boxtimes\C(\Z_p,\eta)\cong\Y(\A)$ as braided fusion category, where $\A$ is a near-group fusion category of type $(\Z_p,p)$, which gives a partial answer to the conjecture of D. Evans and T. Gannon. And we show that there exists a non-trivial $\Z_2$-extension of $\A$ that contains  simple objects of Frobenius-Perron dimension $\frac{\sqrt{p}+\sqrt{q}}{2}$.


\section{Introduction}

A braided spherical fusion category $\C$ is called modular if the $S$-matrix of $\C$ is non-degenerate (see Section \ref{preliminaries}).  Modular fusion category connects  with conformal field theory, quantum groups, representation theory, and mathematical physics, etc \cite{DongLNg,EGNO,KiO,Mu}. Combined with $T$-matrix, which is defined by the ribbon structure $\theta$ of $\C$, these two matrices $(S,T)$ are called the modular data of $\C$. The modular data enjoy many important algebraic and  arithmetic properties. As it was called, the modular data  provides a projective congruence representation $\rho$ of the modular group $\SL(2,\Z)$ of level $N$ \cite{DongLNg,EGNO,NRWW}, where $N=\text{ord}(T)$. Moreover, $\rho$ can be lifted to a linear congruence representation of $\SL(2,\Z)$ of level $n$ with $N\mid n\mid 12N$, that is, it factors through $\SL(2,\Z)\to \SL(2,\Z_n)$, and the linear representation satisfies the  Galois symmetry \cite{DongLNg}.

Finite-dimensional representations of $\SL(2,\Z_n)$ are classified completely in  \cite{Nob,NobWol}. Thus, one could construct (or, reconstruct) modular fusion categories from finite-dimensional congruence representations of $\SL(2,\Z)$, see  \cite{NRWW,NgWZ,Yu2} for applications. In this paper, we are aimed to realize a class of finite-dimensional congruence representations of $\SL(2,\Z)$ as modular representation associated to a modular fusion category. Explicitly, let $p$ be an odd prime, and let $\rho$ be an irreducible $\frac{p+1}{2}$-dimensional representation of $\SL(2,\Z_p)$. Up to isomorphism, it is well-known that there exist just two such representation \cite{Nob}. However, neither of these two representations can be isomorphism to a modular representation associated to a modular fusion category \cite{Eholzer}.  Hence we consider the following question:
\begin{ques}
Let $p<q$ be  odd primes. Is there a modular fusion category $\C$ such that the associated modular representation $\rho_\C\cong\rho_1\oplus\rho_2$, where $\rho_1,\rho_2$ are irreducible representations of dimensional $\frac{p+1}{2}$ and $\frac{q+1}{2}$, respectively?
\end{ques}
When $p=3$ and $q=7$, the answer is positive \cite[Lemma 4.7]{NRWW}. We give a necessary condition on realizing the sum $\rho_1\oplus\rho_2$ in  Theorem \ref{q-p=4}, which states $q-p=4$. Moreover, we show that if such  a modular fusion category $\C$ does exist, then it is connected with a  near-group fusion category $\A$ (see Subsection \ref{subsection3.2}). And we study the structure of $\C$ and the related near-group fusion category $\A$, we  also give a faithful $\Z_2$-extension of $\A$, which generalizes the fusion category $\mathcal{V}$ constructed  by Ostrik in \cite{CMS}.

Since there exists a pointed modular fusion category $\C(\Z_p,\eta)$ of Frobenius-Perron dimension $p$ such that  $\C\boxtimes\C(\Z_p,\eta)\cong\Y(\A)$ as modular fusion category (Theorem \ref{etaleAlg}), which  then can be viewed as evidence that \cite[Conjecture 2]{EG} might be true; and the modular data (of $\C$) obtained in this paper gives a partial solution to the modular data described  with unknown parameters in \cite[Proposition 7]{EG}.

This paper is organized as follows. In Section \ref{preliminaries}, we recall some basic notions and notations of  (modular)  fusion categories, such as  Frobenius-Perron dimension, global dimension, modular data, and the congruence representations of the modular group $\SL(2,\Z)$.  In Section \ref{realization}, we consider the realization of direct sum $\rho_1\oplus\rho_2$ of two irreducible representations of dimension $\frac{p+1}{2}$ and $\frac{q+1}{2}$, respectively. We show in Theorem \ref{q-p=4} that if $\rho_1\oplus\rho_2$ can be realized as representation associated to modular fusion category $\C$ then $q-p=4$. Under the assumption that $\C$ contains a  non-trivial connected \'{e}tale algebra $A$, we prove that $\C_A^0$ is a pointed modular fusion category and $\C_A$ is a near-group fusion category of type $(\Z_p,p)$ in Theorem \ref{etaleAlg} and Theorem \ref{structC_A}. In the last, we construct a faithful $\Z_2$-extension $\M$ of $\C_A$, which contains simple objects of Frobenius-Perron dimension $\frac{\sqrt{p}+\sqrt{q}}{2}$,  and we determine the fusion relations of $\M$ in Corollary \ref{frelation}.
\section{Preliminaries}\label{preliminaries}
In this section,  we will recall some most used  definitions and properties about  modular fusion categories, we refer the readers to \cite{DrGNO,EGNO,ENO,ENO2,Mu} to standard conclusions for  fusion categories and braided fusion categories.
\subsection{Fusion category}
A  $\FC$-linear abelian category $\C$ over the complex number field $\FC$ is called a fusion category if $\C$ is a finite semisimple tensor category \cite{EGNO}. In the following, we use $\Q(\C)$ and $\otimes$ to denote the set of isomorphism classes of simple objects of $\C$ and the monoidal functor on $\C$, respectively.

Let $\C$ be a fusion category. Then its Grothendieck ring is a fusion ring with $\Z_+$-basis $\Q(\C)$ and the multiplication is induced by the monoidal functor $\otimes$. There is a unique homomorphism $\FPdim(-)$, called the Frobenius-Perron homomorphism,  from $\text{Gr}(\C)$ to $\FC$ such that $\FPdim(X)$ is a positive algebraic integer for all non-zero object $X$ \cite{EGNO,ENO}. The sum
\begin{align*}
\FPdim(\C):=\sum_{X\in\Q(\C)}\FPdim(X)^2
\end{align*}
is called the Frobenius-Perron dimension of $\C$.

A fusion category $\C$ is pivotal if it admits a pivotal structure $j$, which is a natural isomorphism from the identity functor $\text{id}$ to the double dual functor $(-)^{**}$ \cite{EGNO}. Then there is a well-defined categorical trace $\text{Tr}(-)$ for all morphism $f\in\Hom_\C(X,X)$, where $X$ is an object of $\C$. Fix a spherical structure $j$ on $\C$, the categorical trace of $\text{id}_X$ is called the categorical dimension of $X$ and denoted by $\dim(X)$, and the sum
\begin{align*}
\dim(\C):=\sum_{X\in\Q(\C)}\dim(X)\dim(X^*)
\end{align*}
is called the global (or quantum) dimension of $\C$. Moreover, the categorical dimension induces a homomorphism from the  Grothendieck ring $\text{Gr}(\C)$ to $\FC$ \cite[Proposition 4.7.12]{EGNO}. If $\dim(X)=\dim(X^*)$ for all objects $X$ of $\C$, then $\C$ is called spherical.

Recall that a fusion ring $R$ is  categorifiable if there exists a fusion category $\C$ such that   $\text{Gr}(\C)=R$ as fusion ring \cite[Definition 4.10.1]{EGNO}, and $\C$ is called a categorification of $R$. For example, for any finite group $G$, pointed fusion category $\vvec_G^\omega$, i.e., the category of  $G$-graded finite-dimensional vector space over $\FC$, is a categorification of the group ring $\Z[G]$, where $\omega\in Z^3(G,\FC^*)$ is a normalized $3$-cocycle on $G$ and $\FC^*:=\FC\backslash\{0\}$.
\subsection{Modular fusion category and modular representation}
A braided fusion category $\C$ is a fusion category with a braiding $c$, which is a natural isomorphism $c_{X,Y}:X\otimes Y\overset{\sim}{\to} Y\otimes X$ satisfying the hexagon equations \cite{EGNO}. In addition, if $\C$ is spherical, then $\C$ is called a pre-modular (or ribbon) fusion category, we use $\theta$ to denote the ribbon structure of $\C$.

Let $\C$ be a pre-modular fusion category. For any simple objects $X,Y$ of $\C$, let $S_{X,Y}:=\text{Tr}(c_{Y,X}c_{X,Y})$, then
\begin{align*}
S=(S_{X,Y}),\quad T=(\delta_{X,Y}\theta_X)
\end{align*}
is called the modular data of $\C$. If the $S$-matrix $S$ is non-degenerate, then $\C$ is said to be a modular fusion category \cite{DrGNO,Mu}. For example, pointed modular fusion categories are in bijective correspondence with metric groups \cite[Proposition 2.41]{DrGNO}. We use $\C(G,\eta)$ to denote the modular fusion category determined by metric group $(G,\eta)$, where $G$ is a finite abelian group and $\eta:G\to\FC^*$ is a non-degenerate quadratic form, the modular data of $\C(G,\eta)$ is
\begin{align*}
S_{g,h}=\frac{\eta(gh)}{\eta(g)\eta(h)}, \theta_g=\eta(g),\forall g,h\in G.
\end{align*}
The $S$-matrix of a modular fusion category $\C$ also satisfies  the Verlinde formula \cite{EGNO}, which states that for any objects $X,Y,Z\in\Q(\C)$,
\begin{align*}\label{Verlinde}
N_{X,Y}^Z:=\dim_\FC(\text{Hom}_\C(X\otimes Y,Z))=\frac{1}{\dim(\C)}\sum_{W\in\Q(\C)}\frac{S_{X,W}S_{Y,W}S_{Z^*,W}}{\dim(W)}.
\end{align*}

Recall that  the modular group $\SL(2,\Z)$  is generated by $\hxs=\left(
                                                      \begin{array}{cc}
                                                        0 & 1 \\
                                                        -1 & 0 \\
                                                      \end{array}
                                                    \right)
$ and $\hxt=\left(
              \begin{array}{cc}
                1 & 1 \\
                0 & 1 \\
              \end{array}
            \right)
$ with relations $\hxs^4=1$ and $(\hxs\hxt)^3=\hxs^2$. The modular data of a modular fusion category $\C$ determines a projective congruence representation $\rho$ of the modular group $\SL(2,\Z)$ of level $N=\text{ord}(T)$ \cite{BNRW,DongLNg,EGNO,NRWW}, that is,  $\ker(\rho)$ kills a congruence subgroup of level $N$, and
 \begin{align*}
\rho:~\hxs\mapsto\frac{1}{\sqrt{\dim(\C)}}S, \hxt\mapsto T,
\end{align*} where
$\sqrt{\dim(\C)}$ is the positive square root of $\dim(\C)$.
Moreover, the projective representation $\rho$  can be lifted to a linear congruence representation $\rho_\C$ of level $n$ and $N\mid n$ by \cite[Theorem II]{DongLNg}, where $n=\text{ord}(\rho_\C(\hxt))$. If $\text{ord}(T)$ is odd, then there is a lifting $\rho'$ of $\rho$ such that $\text{ord}(\rho'(\hxt))=\text{ord}(T)$ \cite[Lemma 2.2]{DongLNg}.

Let $\rho$ be an arbitrary irreducible finite-dimensional congruence representation of $\SL(2,\Z)$ of level $n$, where $n$ is a positive integer. Then  it follows from the Chinese Reminder Theorem that $\rho$ factors through the finite groups
 \begin{align*}
\SL(2,\Z_n)\cong\SL(2,\Z_{p_1^{n_1}})\times\cdots\times\SL(2,\Z_{p_r^{n_r}})
\end{align*} and
$\rho\cong \otimes_{j=1}^r\rho_{p_j}$, where $n=\prod_{j=1}^rp_j^{n_j}$ and $p_j$ are distinct primes, $\rho_{p_j}$ are finite-dimensional representations of subgroups $\SL(2,\Z_{p_j^{n_j}})$. And finite-dimensional irreducible representations of the group $\SL(2,\Z_{p^m})$ are completely classified and constructed explicitly in \cite{Nob,NobWol}.

Hence, one could try to reconstruct of modular fusion categories from finite-dimensional congruence representations of $\SL(2,\Z)$, see  \cite{BNRW,Eholzer,NRWW,NgWZ,Yu2} and the references therein for details. For example, many important properties of modular representations are summarized and characterized in \cite{NRWW}; as an application, modular fusion categories with six simple objects (up to isomorphism) are classified by considering the type of the associated modular representation of $\C$ \cite{NRWW}. A  representation $\rho$ of $\SL(2,\Z)$ is called   realizable if there exists a modular fusion category $\C$ such that $\rho_\C\cong\rho$.
\section{Realization and extension}\label{realization}
In this section, we consider the realization of $\rho_1\oplus\rho_2$ as modular representation associated to a modular fusion category. Under  the assumption that $\rho_1\oplus\rho_2$ can realized as representation of modular fusion category $\C$, we study the structure of $\C$ and show it is related to certain near-group fusion category $\A$. In the last, we construct a faithful $\Z_2$-extension of  $\A$.
\subsection{Realization}
Let $p$ be an odd prime. Let $\rho$ be a $\frac{p+1}{2}$-dimensional irreducible representation of $\SL(2,\Z_p)$. Then \cite[Equation (4.11)]{Eholzer} says
\begin{align*}
\rho(\hxs)=\beta_p\left(
                             \begin{array}{cccc}
                               1 & \sqrt{2} & \cdots & \sqrt{2} \\
                               \sqrt{2} &  \\
                               \vdots & &2\cos(\frac{4\pi ajk}{p})   \\
                               \sqrt{2} &  \\
                             \end{array}
                           \right)=\left(
                                     \begin{array}{cc}
                                       \beta_p & B^T \\
                                       B & D \\
                                     \end{array}
                                   \right),
\rho(\hxt)=(1,T_1),
\end{align*}
where $B^T:=(\sqrt{2}\beta_p,\cdots,\sqrt{2}\beta_p)$ is a $\frac{p-1}{2}$-dimensional vector over $\FC$, $D:=\left(2\beta_p\cos(\frac{4\pi ajk}{p})\right)$ and $T_1:=\text{diag}\left(\zeta_p^a,\cdots,\zeta_p^{a\cdot\left(\frac{p-1}{2}\right)^2}\right)$ are square-matrix of order $\frac{p-1}{2}$, $1\leq j,k\leq \frac{p-1
}{2}$, $\beta_p:=\left(\frac{a}{p}\right)\sqrt{\left(\frac{-1}{p}\right)\frac{1}{p}}$, where $a$ is an integer coprime to $p$ and $\left(\frac{a}{p}\right)$ is the classical Legendre symbol. Notice that $\rho$ is non-degenerate, i.e., the eigenvalues of $\rho(\hxt)$ are multiplicity-free. Given an odd prime $p$, up to isomorphism, it is well-known that there are exactly two such irreducible representations \cite{Nob}, depending on the value $\left(\frac{a}{p}\right)$.

  It was proved  in   \cite{Eholzer} that $\rho$ can't be realized by rational conformal field theory (equivalently, it can't be realized  as modular representation associated  a modular fusion category), as the corresponding fusion rings obtained from the Verlinde formula are not integer valued. However, it was also noted in \cite{Eholzer} that one can obtain an integer valued fusion ring from a direct sum of two such representations for different primes $p,q$ such that $q-p=4$.

  Hence one would like to answer the following question naturally.
\begin{ques}Let $p< q$ be  odd primes. And let $\rho_1$ and $\rho_2$ be irreducible representations of  $\SL(2,\Z_p)$ and $\SL(2,\Z_q)$  such that  $\dim(\rho_1)=\frac{p+1}{2}$ and $\dim(\rho_2)=\frac{q+1}{2}$, respectively. Is $\rho_1\oplus\rho_2$ realizable?
\end{ques}
When $p=3$ and $q=7$, the answer is positive; and $\C$ is a Galois conjugate of modular fusion category $\C(\mathfrak{g}_2,3)$ \cite[Lemma 4.7]{NRWW}. We refer the readers to \cite{BK}  for constructions of modular fusion category $\C(\mathfrak{g},k)$, where $\mathfrak{g}$ is a simple Lie algebra. Notice that if $p=1$ (of course, it is not a prime), and let $\rho_0$ be the trivial representation, then $\rho_0\oplus\rho_2$ is realizable for all primes $q\geq5$, moreover,  the associated modular fusion category $\C$ is Grothendieck equivalent to $\C(\ssl_2,2(q-1))_A^0$ \cite[Theorem 3.12]{Yu2}, where $A$ is the non-trivial \'{e}tale algebra of $\Rep(\Z_2)\subseteq\C(\ssl_2,2(q-1))$ and $\C(\ssl_2,2(q-1))_A^0$ is the core of $\C(\ssl_2,2(q-1))$, see \cite{DMNO,DrGNO,KiO} for details.

In the following theorem, we give a necessary condition to realize $\rho_1\oplus\rho_2$ as modular representation associated to a modular fusion category.

\begin{theo}\label{q-p=4}If there is a  modular fusion category $\C$ such that $\rho_\C\cong\rho_1\oplus\rho_2$, then  $q-p=4$.
\end{theo}
\begin{proof}
It follows from \cite[Theorem 3.23]{NRWW} that
\begin{align*}
\rho_\C(\hxs)=V\left(
         \begin{array}{ccc}
           A & C_1^T & C_2^T \\
           C_1 & D_1 & \textbf{0} \\
           C_2 & \textbf{0} & D_2 \\
         \end{array}
       \right)
V,~ \rho_\C(\hxt)=\left(
                \begin{array}{ccc}
                  E_2 &  &  \\
                   & T_1 &  \\
                   &  & T_2 \\
                \end{array}
              \right)
,
\end{align*}
where $V$ is a signed diagonal orthogonal matrix, $T_1=\text{diag}\left(\zeta_p^{a_1},\cdots,\zeta_p^{a_1\cdot\left(\frac{p-1}{2}\right)^2}\right)$ and $T_2=\text{diag}\left(\zeta_q^{a_2},\cdots,
\zeta_q^{a_2\cdot\left(\frac{q-1}{2}\right)^2}\right)$, and
\begin{align*}
A=U\left(
                                             \begin{array}{cc}
                                               \beta_p &  \\
                                                & \beta_q \\
                                             \end{array}
                                           \right)U^T
=\frac{1}{2}\left(
                                             \begin{array}{cc}
                                               \beta_p+\beta_q &\nu(\beta_p-\beta_q)  \\
                                               \nu(\beta_p-\beta_q) & \beta_p+\beta_q \\
                                             \end{array}
                                           \right)
                                           \end{align*}with $U=\left(
            \begin{array}{cc}
              \frac{1}{\sqrt{2}} &\frac{-\nu}{\sqrt{2}} \\
\frac{\nu}{\sqrt{2}} & \frac{1}{\sqrt{2}} \\
            \end{array}
          \right)
$ and $\nu^2=1$,
\begin{align*}
C_1=(B_1,0)U^T=\beta_p\left(
                                  \begin{array}{cc}
                                    1 & \nu \\
                                    \vdots & \vdots \\
                                    1 & \nu \\
                                  \end{array}
                                \right)
,C_2=(0,B_2)U^T=\beta_q\left(
                                  \begin{array}{cc}
                                    -\nu & 1 \\
                                    \vdots & 1 \\
                                    -\nu & 1 \\
                                  \end{array}
                                \right)
.
\end{align*}

 Let $V=\text{diag}(1,\epsilon_1,\cdots,\epsilon_{\frac{p+q}{2}})$ where $\epsilon_j\in\{\pm1\}$ for all $1\leq j\leq \frac{p+q}{2}$, hence we see
\begin{align*}
\rho_\C(\hxs)&=V\left(
         \begin{array}{ccc}
           A & C_1^T & C_2^T \\
           C_1 & D_1 & \textbf{0} \\
           C_2 & \textbf{0} & D_2 \\
         \end{array}
       \right)
V\\&=V\left(
\begin{array}{cccccccc}
\frac{1}{2}(\beta_p+\beta_q) & \frac{\nu}{2}(\beta_p-\beta_q) & \beta_p & \cdots & \beta_p&-\nu\beta_q&\cdots& -\nu\beta_q\\
\frac{\nu}{2}(\beta_p-\beta_q) &  \frac{1}{2}(\beta_p+\beta_q) & \beta_p\nu
& \cdots  & \beta_p\nu &\beta_q&\cdots&\beta_q\\
\beta_p &\beta_p\nu  &  &&   \\
\vdots& \vdots &  &  &  \\
\beta_p &\beta_p\nu  & & &D_1 & & \textbf{0} \\
-\nu\beta_q&\beta_q&&&\\
\vdots&\vdots&&&\textbf{0}&&D_2\\
 -\nu\beta_q&\beta_q&&&&&\\
\end{array}
\right)V\end{align*}
\begin{align*}
=\left(
\begin{array}{cccccccc}
\frac{1}{2}(\beta_p+\beta_q) & \frac{\nu\epsilon_1}{2}(\beta_p-\beta_q) & \epsilon_2\beta_p & \cdots & \epsilon_\frac{p+1}{2}\beta_p
&-\epsilon_\frac{p+3}{2}\nu\beta_q&\cdots& -\epsilon_\frac{p+q}{2}\nu\beta_q\\
\frac{\nu\epsilon_1}{2}(\beta_p-\beta_q) &  \frac{1}{2}(\beta_p+\beta_q) & \epsilon_1\epsilon_2\beta_p\nu
& \cdots  & \epsilon_1\epsilon_\frac{p+1}{2}\beta_p\nu &\epsilon_1\epsilon_\frac{p+3}{2}\beta_q&
\cdots&\epsilon_1\epsilon_\frac{p+q}{2}\beta_q\\
\epsilon_2\beta_p &\epsilon_1\epsilon_2\beta_p\nu  &  &&   \\
\vdots& \vdots &  &  &  \\
\epsilon_\frac{p+1}{2}\beta_p&\epsilon_1\epsilon_\frac{p+1}{2}\beta_p\nu  && & &V_1D_1V_1 & & \textbf{0} \\
-\epsilon_\frac{p+3}{2}\nu\beta_q&\epsilon_1\epsilon_\frac{p+3}{2}\beta_q&&&\\
\vdots&\vdots&&&&\textbf{0}&&V_2D_2V_2\\
 -\epsilon_\frac{p+q}{2}\nu\beta_q&\epsilon_1\epsilon_\frac{p+q}{2}\beta_q&&&&&\\
\end{array}
\right),
\end{align*}
where $V=\left(
           \begin{array}{cccc}
             1 & &  &  \\
              & \epsilon_1 &  &\\
              &  & V_1 &\\
              &&& V_2
           \end{array}
         \right),
$ $V_1=\left(\begin{array}{cccc}
\epsilon_2&&\\
&\ddots\\
&&\epsilon_\frac{p+1}{2}
\end{array}\right)$,
 $V_2=\left(\begin{array}{cccc}
\epsilon_\frac{p+3}{2}&&\\
&\ddots\\
&&\epsilon_\frac{p+q}{2}
\end{array}\right)$.

Since the categorical dimensions of simple objects are always non-zero,  either the elements in the first or  the second row are are dimensions (multiplied with a non-zero scalar necessarily) of simple objects of $\C$,  depending on which vector represents the unit object.
 We know $\beta_p=\frac{\mu_p}{\sqrt{p}}$ and $\beta_q=\frac{\mu_q}{\sqrt{q}}$, where $\mu_p=\left(\frac{a_1}{p}\right)\sqrt{\left(\frac{-1}{p}\right)}$ and $\mu_q=\left(\frac{a_2}{q}\right)\sqrt{\left(\frac{-1}{q}\right)}$ are $4$-th roots of unity. A  classical theorem of Legendre symbol says  $\left(\frac{a_1}{p}\right)\equiv a_1^\frac{p-1}{2}  \text{mod}~p$, so
 \begin{align*}
 \mu_p=\left(\frac{a_1}{p}\right)\sqrt{\left(\frac{-1}{p}\right)}
=\left\{
\begin{array}{ll}
   \left(\frac{a_1}{p}\right), & \hbox{if $p=4k+1$;} \\
     \left(\frac{a_1}{p}\right)\zeta_4, & \hbox{if $p=4k+3$.}
\end{array} \right.
.\end{align*}
where $\zeta_4$ is a $4$-th primitive root of unity. Notice that
\begin{align*}
|\beta_p+\beta_q|^2
=\frac{(\mu_p\sqrt{q}+\mu_q\sqrt{p})
(\overline{\mu}_p\sqrt{q}+\overline{\mu}_q\sqrt{p})}{pq}
=\frac{(p+q)+(\overline{\mu}_p\mu_q+\overline{\mu}_q\mu_p)\sqrt{pq}}{pq}.
\end{align*}
We claim  $\overline{\mu}_p\mu_q+\overline{\mu}_q\mu_p=2\text{Re}(\overline{\mu}_p\mu_q)=\pm2$. In fact, $\overline{\mu}_p\mu_q+\overline{\mu}_q\mu_p\neq0$, otherwise
\begin{align*}
\dim(\C)=\frac{4}{|\beta_p+\beta_q|^2}=\frac{4pq}{p+q},
\end{align*}
then $p+q$ must contains a prime factor which is coprime to $pq$. However, $\text{ord}(\rho_\C(\hxt))=pq$, it violates the Cauchy's theorem of spherical fusion categories \cite[Theorem 3.9]{BNRW}. Meanwhile, $\overline{\mu}_p\mu_q$ is a $4$-th root of unit, so $2\text{Re}(\overline{\mu}_p\mu_q)=\pm2$,  as claimed. Therefore, \begin{align*}
\dim(\C)=\frac{4}{|\beta_p+\beta_q|^2}=pq\frac{\frac{p+q}{2}\pm\sqrt{pq}}{2},
\end{align*}
 depending on the value of $\text{Re}(\overline{\mu}_p\mu_q)$.
Then \begin{align*}
N\left(\dim(\C)\right)=p^2q^2N\left(\frac{\frac{p+q}{2}\pm\sqrt{pq}}{2}\right)=
p^2q^2\frac{(p-q)^2}{16},
\end{align*}
where $N(\dim(\C))$ and $N\left(\frac{\frac{p+q}{2}\pm\sqrt{pq}}{2}\right)$ are the norms of $\dim(\C)$ and $\frac{\frac{p+q}{2}\pm\sqrt{pq}}{2}$ over $\FQ(\sqrt{pq})$, respectively. Again the Cauchy's theorem of spherical fusion categories \cite[Theorem 3.9]{BNRW}  implies that $\frac{\frac{p+q}{2}\pm\sqrt{pq}}{2}$ must be an algebraic unit in $\FQ(\sqrt{pq})$, that is, $q-p=4$, as desired.
\end{proof}

Below we calculate the dimensions of simple objects of $\C$, denote by $\varepsilon_{pq}:=\frac{\sqrt{p}+\sqrt{q}}{2}$, then $\dim(\C)=pq\varepsilon_{pq}^{\pm2}$.
Since $q-p=4$, we have $\overline\mu_p\mu_q=\left(\frac{a_1}{p}\right)\left(\frac{a_2}{q}\right)=\pm1$. That is, if $a_1$ and $a_2$ are both square residues  or both non-square residues modulo $p$ and $q$, respectively, then $\dim(\C)=pq\varepsilon_{pq}^{-2}$, otherwise, $\dim(\C)=pq\varepsilon_{pq}^2$.

We list the categorical dimensions in both cases explicitly. After  identifying $\Q(\C)$ with the standard basis $\{e_1,\cdots,e_{p+2}\}$ of vector space  $\FC^{p+2}$, the $S$-matrix of $\C$ can be written as
\begin{align*}
S=\left(
\begin{array}{cccccccc}
1& \frac{\nu\epsilon_1(\beta_p-\beta_q)}{\beta_p+\beta_q} & \frac{2\epsilon_2\beta_p}{\beta_p+\beta_q} & \cdots & \frac{2\epsilon_\frac{p+1}{2}\beta_p}{\beta_p+\beta_q}
&\frac{-2\epsilon_\frac{p+3}{2}\nu\beta_q}{\beta_p+\beta_q}&\cdots& \frac{-2\epsilon_\frac{p+q}{2}\nu\beta_q}{\beta_p+\beta_q}\\
\frac{\nu\epsilon_1(\beta_p-\beta_q)}{\beta_p+\beta_q} &  1 & \frac{2\epsilon_1\epsilon_2\beta_p\nu}{\beta_p+\beta_q}
& \cdots  & \frac{2\epsilon_1\epsilon_\frac{p+1}{2}\beta_p\nu}{\beta_p+\beta_q} &\frac{2\epsilon_1\epsilon_\frac{p+3}{2}\beta_q}{\beta_p+\beta_q}&
\cdots&\frac{2\epsilon_1\epsilon_\frac{p+q}{2}\beta_q}{\beta_p+\beta_q} \\
\frac{2\epsilon_2\beta_p}{\beta_p+\beta_q} &\frac{2\epsilon_1\epsilon_2\beta_p\nu}{\beta_p+\beta_q}  &  &&   \\
\vdots& \vdots &  &  &  \\
\frac{2\epsilon_\frac{p+1}{2}\beta_p}{\beta_p+\beta_q}&
\frac{2\epsilon_1\epsilon_\frac{p+1}{2}\beta_p\nu}{\beta_p+\beta_q}  & & &&\frac{2}{\beta_p+\beta_q}V_1D_1V_1  && \textbf{0} \\
\frac{-2\epsilon_\frac{p+3}{2}\nu\beta_q}{\beta_p+\beta_q}&
\frac{2\epsilon_1\epsilon_\frac{p+3}{2}\beta_q}{\beta_p+\beta_q}&&&\\
\vdots&\vdots&&&&\textbf{0}&&\frac{2}{\beta_p+\beta_q}V_2D_2V_2\\
 \frac{-2\epsilon_\frac{p+q}{2}\nu\beta_q}{\beta_p+\beta_q}&
 \frac{2\epsilon_1\epsilon_\frac{p+q}{2}\beta_q}{\beta_p+\beta_q}&&&&&\\
\end{array}
\right),
\end{align*}
\textbf{Case (1)}: $\overline{\mu}_p\mu_q=1$. We can assume that $a_1,a_2$ are both resides modulo $p,q$, respectively, the other case is same. Let $a_1=a_2=1$. Then $\beta_p=\frac{1}{\sqrt{p}}$ and $\beta_q=\frac{1}{\sqrt{q}}$ if $p=4k+1$, $\beta_p=\frac{\zeta_4}{\sqrt{p}}$ and $\beta_q=\frac{\zeta_4}{\sqrt{q}}$ if $p=4k+3$, then $\dim(\C)=pq\varepsilon_{pq}^{-2}$. Let \begin{align*}
&d_1:=\sqrt{q}\varepsilon_{pq}^{-1}=\frac{\sqrt{q}(\sqrt{q}-\sqrt{p})}{2},
~d_1':=d_1\varepsilon_{pq}^2=\sqrt{q}\varepsilon_{pq}=\frac{\sqrt{q}(\sqrt{q}+\sqrt{p})}{2}, \\&~d_2:=\sqrt{p}\varepsilon_{pq}^{-1}=\frac{\sqrt{p}(\sqrt{q}-\sqrt{p})}{2},~ d_2':=d_2\varepsilon_{pq}^2=\sqrt{p}\varepsilon_{pq}=\frac{\sqrt{p}(\sqrt{q}+\sqrt{p})}{2}.
\end{align*} Then the first row of the $S$-matrix  is
 \begin{align*}
 \left(1,\nu\epsilon_1\varepsilon_{pq}^{-2}, \epsilon_2d_1,\cdots,\epsilon_\frac{p+1}{2}d_1,
 -\nu\epsilon_\frac{p+3}{2}d_2,\cdots, -\nu\epsilon_\frac{p+q}{2}d_2\right),
 \end{align*}
and the second row of the  $S$-matrix is
\begin{align*}
 \left(\nu\epsilon_1\varepsilon_{pq}^{-2},1, \epsilon_1\epsilon_2\nu d_1,
 \cdots,\epsilon_1\epsilon_\frac{p+1}{2}\nu d_1,
 \epsilon_1\epsilon_\frac{p+3}{2}d_2,
 \cdots, \epsilon_1\epsilon_\frac{p+q}{2}d_2\right).
 \end{align*}

If the first row are the categorical dimensions of simple objects, that is, the first basis element $e_1$ is the unit object of $\C$, notice that  \begin{align*}
\dim(\C)<\sigma(\dim(\C))=pq\varepsilon_{pq}^2\leq\FPdim(\C),
 \end{align*}
 where $\langle\sigma\rangle=\Gal(\FQ(\sqrt{pq})/\FQ)$, then the second row must be the Frobenius-Perron dimensions of simple objects of $\C$ multiplied by scalar $\nu\epsilon_1\varepsilon_{pq}^{-2}$. Since $\FPdim(X)>0$, $X\in\Q(\C)$,
\begin{align*}
\nu\epsilon_1=\nu\epsilon_\frac{p+3}{2}=\cdots=\nu\epsilon_\frac{p+q}{2}=1, \epsilon_2=\cdots=\epsilon_\frac{p+1}{2}=1,
\end{align*}
consequently, we obtain $\FPdim(\C)=pq\varepsilon_{pq}^2$ and
\begin{align*}
\FPdim(X)\in\left\{1,\varepsilon_{pq}^2,d_1',d_2'\right\},
\dim(X)\in\left\{1,\varepsilon_{pq}^{-2},d_1,
-d_2\right\},\forall X\in\Q(\C).
\end{align*}
It is easy to see that other formal codegrees of $\C$ are either $\frac{\dim(\C)}{d_1^2}=p$ or $\frac{\dim(\C)}{d_2^2}=q$, which can't be the Frobenius-Perron dimension of $\C$ since $\C$ is not pointed, hence $\C$ is a Galois conjugate of a pseudo-unitary fusion category. And the modular data of $\C$ is
\begin{align*}
&S=\left(
\begin{array}{cccccccc}
1& \varepsilon_{pq}^{-2} & d_1 & \cdots & d_1
&-d_2&\cdots& -d_2\\
\varepsilon_{pq}^{-2}  &  1 & d_1
& \cdots  & d_1 &d_2&
\cdots&d_2\\
d_1 &d_1  &  &&   \\
\vdots& \vdots &  &2d_1\cos\left(\frac{4\pi j_1k_1}{p}\right)  & &&\textbf{0} \\
d_1&d_1 & & &&  && \\
-d_2&d_2&&&\\
\vdots&\vdots&&\textbf{0}&&&2d_2\cos\left(\frac{4\pi j_2k_2}{q}\right)\\
 -d_2&d_2&&&&&\\
\end{array}
\right),\\
&T=\left(1,1,\zeta_p,\cdots,\zeta_p^{\left(\frac{p-1}{2}\right)^2},
\zeta_q,\cdots,\zeta_q^{\left(\frac{q-1}{2}\right)^2}\right),
\end{align*}
where $1\leq j_1,k_1\leq\frac{p-1}{2}$ and $1\leq j_2,k_2\leq\frac{q-1}{2}$.

If the second row are the categorical dimensions of simple objects, then $e_2$ is the unit object of $\C$ and the elements in the first row are the Frobenius-Perron dimensions of simple objects multiplied by scalar $\nu\epsilon_1\varepsilon_{pq}^{-2}$, similarly,    \begin{align*}
\nu\epsilon_1=-\nu\epsilon_\frac{p+3}{2}=\cdots=-\nu\epsilon_\frac{p+q}{2}=1, \epsilon_2=\cdots=\epsilon_\frac{p+1}{2}=1,
\end{align*}
again we obtain
\begin{align*}
\FPdim(X)\in\left\{1,\varepsilon_{pq}^2,d_1',d_2'\right\},
\dim(X)\in\left\{1,\varepsilon_{pq}^{-2},d_1,
-d_2\right\},\forall X\in\Q(\C).
\end{align*}
Hence $\FPdim(\C)=pq\varepsilon_{pq}^2$. By using the same argument, we see that $\C$ is a Galois conjugate of a pseudo-unitary fusion category.

\textbf{Case (2)}: $\overline\mu_p\mu_q=-1$.  We can assume $a_1=1$ and $a_2$ is a non-square residue modulo $q$, the other case is same. Then $\beta_p=\frac{1}{\sqrt{p}}$ and $\beta_q=\frac{-1}{\sqrt{q}}$ if $p=4k+1$, $\beta_p=\frac{\zeta_4}{\sqrt{p}}$ and $\beta_q=\frac{-\zeta_4}{\sqrt{q}}$ if $p=4k+3$, moreover, $\dim(\C)=pq\varepsilon_{pq}^2$. The first row of $S$ is
\begin{align*}
\left(1,\nu\epsilon_1\varepsilon_{pq}^2, \epsilon_2d_1',\cdots,\epsilon_\frac{p+1}{2}d_1',
 -\nu\epsilon_\frac{p+3}{2}d_2',\cdots, -\nu\epsilon_\frac{p+q}{2}d_2'\right),
\end{align*}
and the second row of $S$ is
\begin{align*}
\left(\nu\epsilon_1\varepsilon_{pq}^2, 1,\epsilon_1\epsilon_2\nu d_1', \cdots,\nu\epsilon_1\epsilon_\frac{p+1}{2}d_1',
 \epsilon_1\epsilon_\frac{p+3}{2}d_2',\cdots, \epsilon_1\epsilon_\frac{p+q}{2}d_2'\right).
\end{align*}
Notice that $\dim(\C)=pq\varepsilon_{pq}^2$ and $\dim(\C)$ has a Galois conjugate $pq\varepsilon_{pq}^{-2}<\dim(\C)$ and that other formal codegrees of $\C$ are either $p$ or $q$, hence $\FPdim(X)=\dim(X)$ for all simple objects $X$ of $\C$.  Without loss of generality, we can take the elements in the  first row to be the Frobenius-Perron dimensions of simple objects of $\C$, then \begin{align*}
-\nu\epsilon_\frac{p+3}{2}=\cdots=-\nu\epsilon_\frac{p+q}{2}=1, \nu\epsilon_1=\epsilon_2=\cdots=\epsilon_\frac{p+1}{2}=1,
\end{align*}
and
$\FPdim(X)\in\left\{1,\varepsilon_{pq}^2,d_1',d_2'\right\},\forall X\in\Q(\C)$.
In addition,  up to isomorphism, we know that $\C$ contains  $\frac{p-1}{2}$ simple objects of Frobenius-Perron dimension $d_1'$ and $\frac{q-1}{2}$ simple objects of Frobenius-Perron dimension $d_2'$, and  a unique simple object $X$ with $\FPdim(X)=\varepsilon_{pq}^2$. Notice that the modular data of $\C$ is
\begin{align*}
&S=\left(
\begin{array}{cccccccc}
1& \varepsilon_{pq}^2 & d_1' & \cdots & d_1'
&d_2'&\cdots& d_2'\\
\varepsilon_{pq}^2  &  1 & d_1'
& \cdots  & d_1' &-d_2'&\cdots&-d_2'\\
d_1' &d_1'  &  &&   \\
\vdots& \vdots &&2d_1'\cos\left(\frac{4\pi  j_1k_1}{p}\right)  &&  &\textbf{0}  \\
d_1'&d_1' & & &&  && \\
d_2'&-d_2'&&&\\
\vdots&\vdots&&\textbf{0}&&&-2d_2'\cos\left(\frac{4\pi a_2 j_2k_2}{q}\right)\\
 d_2'&-d_2'&&&&&\\
\end{array}
\right),\\
&T=\left(1,1,\zeta_p,\cdots,\zeta_p^{\left(\frac{p-1}{2}\right)^2},
\zeta_q^{a_2},\cdots,\zeta_q^{a_2\left(\frac{q-1}{2}\right)^2}\right),
\end{align*}
where $1\leq j_1,k_1\leq\frac{p-1}{2}$ and $1\leq j_2,k_2\leq\frac{q-1}{2}$.
\begin{coro}\label{pseudounitary}
Let $ \C$ be a modular fusion category such that $\rho_\C\cong\rho_1\oplus\rho_2$, then either $\C$ is a Galois conjugate of a pseudo-unitary fusion category or $\dim(Y)=\FPdim(Y)$ for all simple objects $Y$ of $\C$.
\end{coro}
\begin{prop}
Let $ \C$ be a modular fusion category such that $\rho_\C\cong\rho_1\oplus\rho_2$, then $\C$ must be a simple modular fusion category.
\end{prop}
\begin{proof}On the contrary, assume that  $\C$ contains a non-trivial fusion subcategory $\D$, which must be modular as $\C$ does not contain non-trivial simple objects of integer dimensions, hence $\C\cong\D\boxtimes\D_\C'$ by \cite[Theorem 8.21.4]{EGNO}, where $\D_\C'$ is the centralizer of $\D$ in $\C$. In particular,  \begin{align*}
\rank(\C)=p+3=\rank(\D)\rank(\D_\C').
  \end{align*}
  If $\dim(\D)$  can't be divided by $p$ or $q$, then \cite[Theorem 4.4]{Sch3} says that $\D$ is a non-trivial transitive subcategory in sense of \cite{NgWZ}. Assume $\rank(\D)=\frac{p-1}{2}$ with $p\geq5$, so $\rank(\D_\C')=2+\frac{8}{p-1}$, it is an integer if and only if $p=5$, it is impossible as $9$ is not a prime. Hence both $\dim(\D)$ and $\dim(\D_\C')$ are divided by some primes. Obviously $p$ or $q$ can't divide both $\dim(\D)$ or $\dim(\D_\C')$, we can assume $p\mid\dim(\D)$ and $q\mid\dim(\D_\C')$, then $\dim(\D)=pu_1$ and $\dim(\D_\C')=qu_2$ where $u_j$ are non-trivial algebraic units. Therefore, $\rank(\D)=\frac{p+3}{2}$ and $\dim(\D_\C')=\frac{q+3}{2}$ by \cite[Theorem 3.13]{Yu2}, it is a contradiction.
\end{proof}

%\begin{ques}{\color{red}Does $A=I\oplus X$ admit an \'{e}tale algebra structure when $p\geq7$?} By  decomposing  $\FPdim(X)^2$ into sums of Frobenius-Perron dimensions of simple objects, it seems that $X\otimes X=\oplus_{Y\in\Q(\C)}Y$, i.e., the multiplicity of each simple object is $1$. If so, then $A\otimes A=2I\oplus 3X\oplus_{Y\in\Q(\C),Y\ncong I,X}Y$, {\color{cyan}how to determine the braiding of $c_{X,X}$ and show $A$ admits a commutative algebra structure} %(follow \cite{KiO} result on \'{e}tale algebras?)?\end{ques}

Let $\C$ be a braided fusion category.  Recall that a commutative algebra $A$ in $\C$ is said to be a connected \'{e}tale  algebra if the category $\C_A$ of right $A$-modules in $\C$ is semisimple and   $\Hom_\C(I,A)=\FC$  \cite[Definition 3.1]{DMNO}. Let $(M,\mu_M)\in\C_A$, where $\mu_M:M\otimes A\to M$ is the right $A$-module morphism of $M$. Then $M$ is a local (or dyslectic) module if $\mu_M=\mu_M\circ (c_{A,M}c_{M,A})$ \cite{DMNO,KiO}, where $c$ is the braiding of $\C$. The category of local modules over a connected \'{e}tale algebra $A$  is a braided fusion category, which will be denoted by $\C_A^0$ below.
\begin{theo}\label{etaleAlg}
Let $\C$ be a modular fusion category  such that $\rho_\C\cong\rho_1\oplus\rho_2$. If $\C$ contains a non-trivial connected \'{e}tale algebra $A$, then  $\C_A^0$ is a pointed modular fusion category of dimension $p$. In particular,  $\C$ can't be braided equivalent to the Drinfeld center of a fusion category.
\end{theo}
\begin{proof}
As we noticed in Corollary \ref{pseudounitary}, we have $\dim(Y)=\FPdim(Y)$ for all objects $Y$ of $\C$ or $\C$ is a Galois conjugate of a pseudo-unitary fusion category. After replacing $\C$ by its Galois conjugate (if necessary), we know the Frobenius-Perron dimensions of objects coincide with the categorical dimensions of  objects.

Let $A$ be a non-trivial connected \'{e}tale algebra of $\C$.
In a pseudo-unitary fusion category, we know  any connected \'{e}tale algebra have trivial twist \cite[Lemma 2.2.4]{Sch1}. Meanwhile,  the modular fusion category $\C$ contains  only two simple objects $\{I,X\}$ (up to isomorphism)  with trivial twist, and
  the categorical dimension of   $X$ is $\varepsilon_{pq}^2$. Therefore, $A=I\oplus nX$ for some $n\geq1$. Since $\dim(\C)=pq\varepsilon_{pq}^2$ and $\dim(A)=1+n\varepsilon_{pq}^2$, so $\frac{\dim(\C)}{\dim(A)^2}$ is an algebraic integer.
Notice that
\begin{align*}
N\left(\frac{\dim(\C)}{\dim(A)^2}\right)=\frac{p^2q^2}{(n^2+1+n\frac{p+q}{2})^2},
 \end{align*}
 hence $1+n^2+n\frac{p+q}{2}=q$ as $\frac{p+q}{2}>p$. Then $n\leq 1$, otherwise $n\frac{p+q}{2}\geq q$, it is impossible.

 Thus, $A=I\oplus X$,  and \cite[Remark 3.4]{DMNO} states that it is a $\C$-rigid algebra in sense of \cite{KiO}. Then it follows from   \cite[Theorem 4.5]{KiO} that $\C_A^0$ is a  modular fusion category and
\begin{align*}
\dim(\C_A^0)=\frac{\dim(\C)}{\dim(A)^2}=\frac{pq\varepsilon_{pq}^2}{(1+\varepsilon_{pq}^2)^2}=p,
\end{align*}
which must be pointed by \cite[Theorem 5.12]{Sch3}.  Moreover,
 \begin{align*}
\C\boxtimes(\C_A^0)^\text{rev}\cong\Y(\C_A)
\end{align*}
as modular fusion categories \cite[Corollary 3.30]{DMNO}, where $(\C_A^0)^\text{rev}=\C_A^0$ as fusion category but with reverse braiding \cite{EGNO}.
Thus  \cite[Lemma 5.9]{DMNO} says that $\C$ is Witt equivalent to $\C(\Z_p,\eta)$, whose Witt equivalence class is non-trivial, so $\C$ can't be braided tensor equivalent to the Drinfeld center of any spherical fusion category by \cite[Proposition 5.8]{DMNO}.
\end{proof}

\begin{rema}\label{Subalgebra}
As we all know, there is a  conformal embedding $G_{2,3}\subseteq E_{6,1}$ \cite[Appendix]{DMNO}, so modular fusion category $\C(\mathfrak{g}_2,3)$  contains a non-trivial \'{e}tale algebra $A$ such that there is a braided equivalence $\C(\mathfrak{g}_2,3)_A^0\cong\C(\mathfrak{e}_6,1)$, which is braided equivalent to $\C(\Z_3,\eta)$ \cite[Proposition A.4.1]{CMS}. Note $\dim(A)=\frac{7+\sqrt{21}}{2}=1+\varepsilon_{21}^2$, hence $A=I\oplus X$ by Theorem \ref{etaleAlg}.

However, when $p>3$, we don't know whether there always exists an \'{e}tale algebra structure on the object $I\oplus X$ currently. We believe the answer is positive.
\end{rema}
\begin{rema}
Let $\I:\C_A\to \Y(\C_A)$ be the right adjoint functor to the forgetful functor $F:\Y(\C_A)\to\C_A$. Then simple direct summand of $\I(I)$ all have trivial twists by
\cite[Theorem 4.1]{NS}. Let $Z_j$ ($1\leq j\leq \frac{p-1}{2}$) be the simple objects of $\C$ such that $\FPdim(Z_j)=\frac{\sqrt{q}(\sqrt{p}+\sqrt{q})}{2}$, then $\theta_{Z_j}$ are primitive $p$-th roots of unity. Let $g$ be a generator of $\Z_p$. Then \begin{align*}
\theta_{Z_j}^{-1}=\theta_{g^{k_j}}=\theta_{g^{-k_j}}
 \end{align*}for a unique $k_j$ with $1\leq k_j\leq\frac{p-1}{2}$. Hence, up to isomorphism,   $\Y(\C_A)=\C\boxtimes(\C_A^0)^\text{rev}$ has exactly $p+1$ simple objects with trivial twist, which are \begin{align*}
\left\{I\boxtimes I, X\boxtimes I, Z_j\boxtimes g^{k_j},Z_j\boxtimes g^{-k_j}\bigg| 1\leq j\leq\frac{p-1}{2}\right\}.
\end{align*}
Indeed, in the next subsection, we will show that the Grothendieck ring $\text{Gr}(\C_A)$ is commutative (see Theorem \ref{structC_A}), therefore, $\I(I)$ must be multiplicity-free by \cite[Corollary 2.16]{O2}, and these objects are exactly all the direct summands of $\I(I)$.
\end{rema}

\subsection{The structure of fusion category $\C_A$}\label{subsection3.2}

In this subsection, we show that the category $\C_A$ obtained in Theorem \ref{etaleAlg} is a near-group fusion category of type $(\Z_p,p)$.

Let $G$ be a finite group, $\Z_+:=\Z_{\geq0}$ and $n\in\Z_+$. Recall that a fusion ring $R$ with $\Z_+$-basis $\{g|g\in G\}\cup\{X\}$ is called a near-group fusion ring of type $(G,n)$ \cite{Si} if
\begin{align*}
gX=Xg=X,~XX=\sum_{g\in G}g+nX.
\end{align*}
When $n=0$, it is well-known that $R$ is categorifiable if and only if $G$ is an abelian group,  the corresponding fusion categories are called Tambara-Yamagami fusion categories, which is completely classified in \cite{TY}. We denote these fusion categories by $\mathcal{TY}(G,\tau,\mu)$, where $\tau$ is a  non-degenerate bi-character on $G$ and $\mu$ is a square root of $|G|^{-1}$.
\begin{theo}\label{structC_A}$\C_A$ is a near-group fusion category of type $(\Z_p,p)$.
\end{theo}
\begin{proof}As we have a braided tensor equivalence  $\C\boxtimes(\C_A^0)^\text{rev}\cong\Y(\C_A)$ by Theorem \ref{etaleAlg},  then $\dim(\C_A)=p\sqrt{q}\varepsilon_{pq}=\frac{p(\sqrt{pq}+q)}{2}$, whose Galois conjugate is $\frac{p(-\sqrt{pq}+q)}{2}$. It was proved that fusion category $\C_A$ is faithfully graded by the following Galois group \begin{align*}
\Gal(\FQ(\FPdim(Y):Y\in\Q(\C_A)))/\FQ(\FPdim(\C_A))),
 \end{align*}
 which  is an elementary abelian $2$-group \cite[Proposition 1.8]{GaSch}, so the order of the Galois group is a factor of $\FPdim(\C_A)$ by \cite[Theorem 3.5.2]{EGNO}. Since $2\nmid\FPdim(\C_A)$,
we see \begin{align*}
\FQ(\FPdim(Y):Y\in\Q(\C_A))=\FQ(\sqrt{pq})= \FQ(\varepsilon_{pq}^2).
\end{align*}
 Notice that  \cite[Proposition 8.15]{ENO} says the ratio
 $\frac{\FPdim(\C_A)}{\FPdim((\C_A)_\text{int})}$ is an algebraic integer, where $(\C_A)_\text{int}$ is the maximal integral fusion subcategory of $\C_A$,
 so the only prime factor of $\FPdim((\C_A)_\text{int})$ is $p$, as $\C_A^0$ is pointed by Theorem \ref{etaleAlg}. Hence, $(\C_A)_\text{int}=\C_A^0$.

Let $Z$  be an arbitrary non-invertible simple object of $\C_A$ such that $\FPdim(Z)=\frac{a+b\sqrt{pq}}{2}$, which is an algebraic integer, where $a,b$ are rational with $b\neq0$. Then the minimal polynomial of $\FPdim(X)$ is \begin{align*}
x^2-(\FPdim(Z)+\sigma(\FPdim(Z))x+\FPdim(Z)\sigma(\FPdim(Z)),
 \end{align*}
 where $\sigma(\sqrt{pq})=-\sqrt{pq}$. Note that $\FPdim(Z)+\sigma(\FPdim(Z))=a\in\FQ$, so $a$ is an integer. And $m:=\FPdim(Z)\sigma(\FPdim(Z))=\frac{a^2-b^2pq}{4}$ is also an integer, then $b^2pq=a^2-4m\in\Z$. Assume $b=\frac{r}{s}$ where $(r,s)=1$, notice that $(pq,s)=1$, otherwise $p$ or $q$ is a factor of $(r,s)$, it is a contradiction. So  $b\in\Z$.

Then $\FPdim(Z)^2=\frac{\frac{a^2+b^2pq}{2}+ab\sqrt{pq}}{2}$, while
\begin{align*}
\FPdim(\C_A)=\frac{p (q+\sqrt{pq})}{2}=\sum_{Y\in\Q(\C_A)}\FPdim(Y)^2\geq \FPdim(\C_A^0)+\FPdim(Z)^2,
\end{align*}
so $\FPdim(Z)^2=\frac{\frac{a^2+b^2pq}{2}+ab\sqrt{pq}}{2}\leq\frac{p(q-2)+p\sqrt{pq}}{2}$, by comparing the rational and irrational parts, we obtain that  $b^2\leq 1$, consequently $b=1$ ($b\neq-1$, otherwise $\FPdim(Z)$ has a Galois conjugate whose absolute value is strictly larger than $\FPdim(X)$, which is impossible \cite[Theorem 3.2.1]{EGNO}). Therefore, up to isomorphism, $\C_A$  has exactly one non-invertible  simple object  $Z$. Since
\begin{align*}
\FPdim(\C_A)=p\sqrt{q}\varepsilon_{pq}=p+\left(\frac{p+\sqrt{pq}}{2}\right)^2,
\end{align*}
 $\FPdim(Z)=\frac{p+\sqrt{pq}}{2}$. By comparing the Frobenius-Perron dimensions of simple objects, we see
\begin{align*}
 Z\otimes Z=\oplus_{g\in\Z_p}g\oplus pZ,
\end{align*}
i.e., $\C$ is  a near-group fusion category of type $(\Z_p,p)$.
\end{proof}

\begin{rema}It worths to note that the categorifications of  near-group fusion rings was characterized with complicate linear and nonlinear equations by using Cuntz algebra, see \cite{Iz} and the reference therein for details. Conclusions from \cite{EG,Iz} suggest that  there may exists an infinite family of near-group fusion categories of type $(G,|G|)$, where $G$ is an abelian group. However, in order to  show there such a near-group fusion category exists, one need to solve these equations, which is a non-trivial task, see \cite[Appendix A]{Iz} for solutions of groups of small orders. With the help of computers, when $|G|\leq 13$, the answer is affirmative \cite[Proposition 6]{EG}, and recently this result is improved for cyclic groups of order less than $31$ in \cite{Bud}.

Moreover,  for an  an arbitrary abelian group $G$ of odd order, let  $\A$ be a near-group fusion category of type $(G,|G|)$,  it was conjectured in \cite[Conjecture 2]{EG} that
\begin{align*}
 \Y(\A)\cong\C\boxtimes\C(G,\eta_1)
 \end{align*}
 as modular fusion category, we refer the readers to \cite[Proposition 7]{EG} \cite[Theorem 6.8]{Iz} for a detailed description of the modular data of $\C$.
\end{rema}
Notice that $\A$ contains a unique non-trivial fusion subcategory $\vvec_{\Z_p}$, so $\I(I)$ contains a unique non-trivial \'{e}tale subalgebra $A$ such that $\Y(\A)_A^0\cong\Y(\vvec_{\Z_p})$ as braided fusion category and $\FPdim(A)=\frac{\dim(\A)}{p}=\sqrt{q}\varepsilon_{pq}$ by \cite[Theorem 4.10]{DMNO}. By comparing the Frobenius-Perron dimensions of simple objects, we know $A=I\oplus X$, see Remark \ref{Subalgebra}.

  It was also conjecture in \cite{EG} that  the modular data of $\Y(\A)$ is determined by metric groups $(G,\eta_1)$ and $(H,\eta_2)$, where $H$ is an abelian group of order $|G|+4$. Indeed, if we require $\alpha=\beta=1$, where $\alpha,\beta$ are parameters  in
\cite[Proposition 7]{EG}, it is easy to see that the modular data $\mathcal{MD}_{G,H}(\eta_1,\eta_2)$ of \cite{EG} is exactly  that of $\C$ in the pseudo-unitary situation. Hence, under the assumption that $\C$ contains a  non-trivial \'{e}tale algebra,  Theorem \ref{etaleAlg} gives a partial positive answer to \cite[Conjecture 2]{EG} and provides solutions to the conjectured modular data  of $\C$, and  our result suggests  that the conjecture  might be true.

Based on conclusions of categorification of near-group fusion rings, we propose the following conjecture, and we believe there is an affirmative answer.
 \begin{conj}Let $p,q,\rho_1,\rho_2$ be the notations as before. Then there exists  a modular fusion category $\C$ such that $\rho_\C\cong\rho_1\oplus\rho_2$ if and only if $q-p=4$.
 \end{conj}

\subsection{A faithful $\Z_2$-extension of $\C_A$}
In this subsection, we provide a faithful $\Z_2$-extension $\M$ of the near-group fusion category $\C_A$. In particular, we prove that $\M$ contains simple objects of Frobenius-Perron dimension $\frac{\sqrt{p}+\sqrt{q}}{2}$. In the last, we construct a class of  non-commutative fusion rings that are non-trivial $\Z_2$-extensions of near-group fusion rings of type $(\Z_n,n)$ for all $n\geq1$.

For any odd prime $p$, note that there is a modular fusion category of Frobenius-Perron dimension $4p$, which is braided tensor equivalent to a $\Z_2$-equivariantization of a Tambara-Yamagami fusion category $\mathcal{TY}(\Z_p,\tau,\mu)$ \cite[Proposition 5.1]{GNN}, we refer the readers to \cite{DrGNO,EGNO} for definition and properties of equivariantization and de-equivariantization of fusion categories by finite groups. Moreover, modular data of  $\mathcal{TY}(\Z_p,\tau,\mu)^{\Z_2}$ are given in \cite[Example 5D]{GNN} explicitly. In particular,  $\D$ contains a Tannakian fusion subcategory $\Rep(\Z_2)$ and two simple objects of Frobenius-Perron dimension $\sqrt{p}$.

Let $\rho'$ be a $3$-dimensional irreducible congruence representation of $\SL(2,\Z)$ of level $4$ with
\begin{align*}
\rho'(\hxs)=\mu_p\left(
             \begin{array}{ccc}
               0 & \frac{1}{\sqrt{2}} & \frac{1}{\sqrt{2}} \\
               \frac{1}{\sqrt{2}} & \frac{-1}{2} & \frac{1}{2} \\
               \frac{1}{\sqrt{2}} & \frac{1}{2} & \frac{-1}{2} \\
             \end{array}
           \right), \rho'(\hxt)=\text{diag}(1,\xi_1,-\xi_1),
\end{align*}
where $\beta_p=\mu_p\frac{1}{\sqrt{p}}$, $\xi_1$ is a square root of the central charge $\xi$ (or $-\xi$) of  $\C(\Z_p,\eta)$ \cite{GNN}.
\begin{prop}
Let $p\geq3$ be an odd prime, and let $\rho_1$ be an irreducible representation of dimension $\frac{p+1}{2}$ of $\SL(2,\Z_p)$. If $\rho_\C\cong\rho_1\oplus\rho'$, then   $\C$ is braided equivalent to a $\Z_2$-equivariantization of $\mathcal{TY}(\Z_p,\tau,\mu)$.
\end{prop}
\begin{proof} (\textit{sketched})
Since $\rho_1$ and $\rho'$ are non-degenerate,  it follows from \cite[Theorem 3.23]{NRWW} (see also Theorem \ref{q-p=4}) that
\begin{align*}
\rho_\C(\hxs)
&=\left(
\begin{array}{cccccccc}
\frac{1}{2}\beta_p & \frac{\nu\epsilon_1}{2}\beta_p & \epsilon_2\beta_p & \cdots & \epsilon_\frac{p+1}{2}\beta_p&-\epsilon_\frac{p+3}{2}\frac{\nu\mu_p}{2}
&-\epsilon_\frac{p+5}{2}\frac{\nu\mu_p}{2}\\
\frac{\nu\epsilon_1}{2}\beta_p &  \frac{1}{2}\beta_p & \epsilon_1\epsilon_2\beta_p\nu
& \cdots  & \epsilon_1\epsilon_\frac{p+1}{2}\beta_p\nu &\epsilon_1\epsilon_\frac{p+3}{2}\frac{\nu\mu_p}{2}&
\epsilon_1\epsilon_\frac{p+5}{2}\frac{\nu\mu_p}{2}\\
\epsilon_2\beta_p &\epsilon_1\epsilon_2\beta_p\nu  &  &&   \\
\vdots& \vdots &  &  &  \\
\epsilon_\frac{p+1}{2}\beta_p&\epsilon_1\epsilon_\frac{p+1}{2}\beta_p\nu  &&  &V_1D_1V_1 & & \textbf{0} \\
\epsilon_\frac{p+1}{2}\nu\beta_p&\epsilon_1\epsilon_\frac{p+1}{2}\nu\beta_p&&&\\
-\epsilon_\frac{p+3}{2}\frac{\nu\mu_p}{2}
&\epsilon_1\epsilon_\frac{p+3}{2}\frac{\nu\mu_p}{2}&&&\textbf{0}
&&V_2D_3V_2\\-\epsilon_\frac{p+5}{2}\frac{\nu\mu_p}{2}&
\epsilon_1\epsilon_\frac{p+5}{2}\frac{\nu\mu_p}{2}&&&&&\\
\end{array}
\right),\\
\rho_\C(\hxt)&=(1,1,\zeta_p^a,\cdots,\zeta_p^{a\left(\frac{p-1}{2}\right)^2},\xi_1,-\xi_1),
\end{align*}
where $\nu,\epsilon_1,\cdots,\epsilon_\frac{p+5}{2}\in\{\pm1\}$, and
 \begin{align*}
V_1=\left(
      \begin{array}{ccc}
        \epsilon_2 &  &  \\
         & \ddots&  \\
         &  & \epsilon_\frac{p+1}{2} \\
      \end{array}
    \right)
,V_2=\left(
 \begin{array}{cc}
  \epsilon_\frac{p+3}{2} &  \\
   & \epsilon_\frac{p+5}{2}\\
   \end{array}
   \right), D_3=\mu_p\left(
    \begin{array}{cc}
     \frac{-1}{2} & \frac{1}{2} \\
     \frac{1}{2} & \frac{-1}{2} \\
     \end{array}
   \right).\end{align*}
  Same as Theorem \ref{q-p=4}, if we  identify the $\Q(\C)$ with the standard basis of vector space, then
  \begin{align*}
  S=\left(
\begin{array}{cccccccc}
1&  \nu\epsilon_1  & 2\epsilon_2 & \cdots & 2\epsilon_\frac{p+1}{2}
&-\epsilon_\frac{p+3}{2}\nu\sqrt{p}&-\epsilon_\frac{p+5}{2}\nu\sqrt{p}\\
\nu\epsilon_1 &  1 & 2\epsilon_1\epsilon_2\nu
& \cdots  & 2\epsilon_1\epsilon_\frac{p+1}{2}\nu &\epsilon_1\epsilon_\frac{p+3}{2}\nu\sqrt{p}&
\epsilon_1\epsilon_\frac{p+5}{2}\nu\sqrt{p}\\
2\epsilon_2 &2\epsilon_1\epsilon_2\nu  &  &&   \\
\vdots& \vdots &  & &4\epsilon_j\epsilon_k\cos\left(\frac{4\pi a jk}{p}\right) && \textbf{0}\\
2\epsilon_\frac{p+1}{2}&2\epsilon_1\epsilon_\frac{p+1}{2}\nu&&&\\
-\epsilon_\frac{p+3}{2}\nu\sqrt{p}&\epsilon_1\epsilon_\frac{p+3}{2}\nu\sqrt{p}&&&
&-\sqrt{p}&\epsilon_\frac{p+3}{2}\epsilon_\frac{p+5}{2}\sqrt{p}\\
-\epsilon_\frac{p+5}{2}\nu\sqrt{p}&\epsilon_1\epsilon_\frac{p+5}{2}\nu\sqrt{p}&&&\textbf{0}&
\epsilon_\frac{p+3}{2}\epsilon_\frac{p+5}{2}\sqrt{p}&-\sqrt{p}\\
\end{array}
\right),
  \end{align*}
we know $\dim(\C)=\FPdim(\C)=4p$, so $\C$ is pseudo-unitary. If the first row are Frobenius-Perron dimensions of simple objects, then \begin{align*}
\epsilon_\frac{p+3}{2}=\epsilon_\frac{p+5}{2}=-1,
\nu=\epsilon_1=\epsilon_2=\cdots=\epsilon_\frac{p+1}{2}=1;
  \end{align*}
  if the second row are Frobenius-Perron dimensions of simple objects, then \begin{align*}
  \nu=\epsilon_1=\epsilon_2=\cdots=\epsilon_\frac{p+3}{2}=\epsilon_\frac{p+5}{2}=1.
  \end{align*}
  From both cases, we know that $\C$ always contains a non-trivial Tannakian fusion subcategory $\Rep(\Z_2)$, hence its core $\C_{\Z_2}^0$ is a pointed modular fusion category of Frobenius-Perron dimension $p$ \cite[Corollary 3.32]{DMNO}. Since $\C$ is not integral,  $\C_{\Z_2}$ must be a Tambara-Yamagami fusion category $\mathcal{TY}(\Z_p,\tau,\mu)$, hence $\C\cong\mathcal{TY}(\Z_p,\tau,\mu)^{\Z_2}$ \cite{DrGNO,GNN}, as desired.
\end{proof}
We note that there exists a modular fusion category $\C$, which is also obtained from $\Z_2$-equivariantization of a Tambara-Yamagami fusion category of dimension $2p$, but $\rho_\C\ncong\rho_1\oplus\rho'$; when $p=5$, see \cite[Theorem 4.15]{NRWW} for details.

\begin{theo}\label{categoryM}
There is a fusion category $\M$ which is a faithful $\Z_2$-extension of $\C_A$ and a non-degenerate fusion category $\D$ such that $\C\boxtimes\D\cong\Y(\M)$; moreover, $\M$ contains exactly $p$ simple objects of Frobenius-Perron dimension $\varepsilon_{pq}$.
\end{theo}

\begin{proof}
Indeed, let $\D=\mathcal{TY}(\Z_p,\tau,\mu)^{\Z_2}$ such that \begin{align*}\D_{\Z_2}^0\cong\C(\Z_p,\eta^{-1})\cong\C(\Z_p,\eta)^\text{rev},
 \end{align*}
 where $\eta^{-1}(g):=\eta(g)^{-1}$ for all $g\in\Z_p$. Consequently, we have braided equivalences
 \begin{align*}
(\C\boxtimes\D)_{\Z_2}^0\cong\C\boxtimes\D_{\Z_2}^0\cong\C\boxtimes\C(\Z_p,\eta^{-1})\cong\Y(\C_A),
 \end{align*}
by Theorem \ref{etaleAlg}, so $\C\boxtimes\D\cong\Y(\M)$ with fusion category $\M$ being a faithful $\Z_2$-extension of the near-group fusion category $\C_A$ by \cite[Theorem 1.3]{ENO2}.

Let $\M=\oplus_{h\in\Z_2}\M_h$ with $\M_e=\C_A$.  Since $\Y(\M)$ contains a simple object of Frobenius-Perron dimension $\sqrt{p}$, $\M$ contains an object $M$ of Frobenius-Perron dimension  $\sqrt{p}$. We claim  that $M\in\M_h$. Indeed, assume $M=M_1\oplus \M_2$ with $M_1\in\M_e$ and $M_2\in\M_h$, respectively, then $\FPdim(M_i)\in\FQ(\sqrt{p})$ by \cite[Lemma 1.1]{GaSch}. Meanwhile, $\FPdim(Z)\in\FQ(\sqrt{pq})$ for all simple objects $Z$ of $\C_A$, so $\FPdim(M_1)$ must be an integer and $\FPdim(M_2)=\sqrt{p}-\FPdim(M_1)$.  If $M_1$ is a non-zero object, then $\FPdim(M_1)\geq1$,  which implies $\FPdim(M_2)$ admits a Galois conjugate whose absolute value is strictly larger than $\FPdim(M_2)$, it is impossible by \cite[Theorem 3.2.1]{EGNO}.  Hence, $M=M_2\in\M_h$, as claimed.

Since $\M$ is $\Z_2$-graded, $M\otimes M\in\M_e$. Notice that $M\otimes M$ must be a direct sum of integral simple objects of $\M_e$, so $M\otimes M=\oplus_{g\in\Z_p}g$. Hence, $M$ is simple and self-dual. Let $\B$ and $\M_\text{int}$ be the maximal weakly integral and integral fusion subcategories of  $\M$, respectively, then $\B$ is faithfully graded by an elementary abelian $2$-group $G$ with $\M_\text{int}$ being the trivial component \cite[Proposition 3.5.7]{EGNO}. Therefore,  $\FPdim(\B)=p|G|$ by \cite[Theorem 3.5.2]{EGNO},  and $\FPdim(\B)$ is a factor of $\FPdim(\M)$ \cite[Proposition 8.15]{ENO}, so $G=\Z_2$. In particular, $\M$ has a unique simple object $M$ of Frobenius-Perron dimension $\sqrt{p}$.

Let $Y\in\Q(\M_h)$ be an arbitrary simple object satisfying $Y\ncong M$, then $M\otimes Y\in\M_e$. Obviously, $g$ can't be a direct summand of $M\otimes Y$ for all invertible objects $g$ of $\M_e$. Therefore, there exists a positive integer $n_Y$ such that $M\otimes Y=n_YX$, then
\begin{align*}
\FPdim(Y)=\frac{n_Y\FPdim(X)}{\FPdim(M)}
=\frac{n_Y\sqrt{p}\varepsilon_{pq}}{\sqrt{p}}=n_Y\varepsilon_{pq}.
 \end{align*}Notice that $Y\otimes Y^*\in\M_e$, so $Y\otimes Y^*$ is a direct sum of simple objects of $\M_e$.
If
\begin{align*}
Y\otimes Y^*=\oplus_{g\in\Z_p}g\oplus m_YX
\end{align*}
 for some positive integer $m_Y$, as $q=p+4$, then
 \begin{align*}
&\FPdim(Y)^2=n_Y^2\varepsilon_{pq}^2=\frac{(p+2)n_Y^2+n_Y^2\sqrt{pq}}{2}\\
&=p+m_Y\sqrt{p}\varepsilon_{pq}=\frac{(2+m_Y)p+m_Y\sqrt{pq}}{2}.
\end{align*}
By comparing the rational and irrational parts of the above equation, we obtain $n_Y^2=m_Y$ and $p=n_Y^2$, which is absurd. Therefore, $Y\otimes Y^*=I\oplus m_YX$, then previous argument also implies $m_Y=n_Y=1$. In particular, for any non-trivial invertible  object $g$, we have $g\otimes Y\ncong Y$, hence the $\Z_2$-grading of $\M$ induces a transitive action of $\Z_p$ on $\Q(\M_h)$. Up to isomorphism, $\M_h$ contains at  least $p$ non-isomorphic simple objects $\{Y_j\}_{j=1}^p$ of Frobenius-Perron dimension $\varepsilon_{pq}$ and a unique simple object of Frobenius-Perron dimension $\sqrt{p}$. Then
 \begin{align*}
 \FPdim(\M_e)=\FPdim(\M_h)\geq p\FPdim(Y_j)^2+\FPdim(M)^2 =p\varepsilon_{pq}^2+p=\FPdim(\M_e),
 \end{align*}
 thus $\Q(\M_h)=\{M\}\cup\{Y_j|1\leq j\leq p\}$.
\end{proof}

\begin{coro}\label{frelation}
Let $\M$ be the $\Z_2$-extension of $\C_A$, and let $Y$ be an arbitrary simple object of Frobenius-Perron dimension $\varepsilon_{pq}$. Then  the fusion rules of $\M$ is given by the following relations
\begin{align*}
&X\otimes X=\oplus_{g\in\Z_p}g\oplus pX,g^i\otimes g^j=g^{i+j}, g\otimes X=X\otimes g=X,M\otimes M=\oplus_{g\in\Z_p}g,\\
&g^jY:=g^j\otimes Y=Y\otimes g^{p-j}, M\otimes g^jY=X=g^jY\otimes M,X\otimes M=M\otimes X=\oplus_{j=1}^pg^jY, \\
&X\otimes g^jY=g^jY\otimes X=M\oplus \oplus_{j=1}^p g^jY,~ g^jY\otimes g^kY=g^{j+p-k}\oplus X.
\end{align*}
In particular, non-invertible simple objects  of $\M$ are self-dual.
\end{coro}
\begin{proof}
Let $Y$ be a simple object of $\M$ of Frobenius-Perron dimension $\varepsilon_{pq}$. As $\Q(\M_h)$ contains $p$ simple objects of same Frobenius-Perron dimension, without loss of generality, we can choose $Y$ to be self-dual, and it follows from Theorem \ref{categoryM} that $Y\otimes Y=I\oplus X$.

Let $g$ be a non-invertible simple object. Then there exists a unique  $1\leq k\leq p-1$ such that $ g\otimes Y\cong Y\otimes g^k$.  Consequently,
\begin{align*}
\FC&=\Hom_\M(g\otimes Y,Y\otimes g^k)\cong\Hom_\M(g,Y\otimes g^k\otimes Y)\\&\cong\Hom_\M(g,Y\otimes Y\otimes g^{k^2})=\Hom_\M(g,g^{k^2}),
\end{align*}
which  means $k^2\equiv1~\text{mod}~ p$, then $k=1,p-1$.

If $g\otimes Y=Y\otimes g$, then $g^jY:=g^j\otimes Y=Y\otimes g^j$ for all $1\leq j\leq p$. As $X\otimes g^j=X$,
\begin{align*}
\Hom_\M(X\otimes g^jY,g^kY)&\cong\Hom_\M(X\otimes Y,g^kY)\\
&\cong \Hom(X,g^kY\otimes Y)\\
&\cong\Hom_\M(X,g^k\oplus X)=\FC
\end{align*}
for all $1\leq j,k\leq p$, we see $\oplus_{k=1}^pg^kY\subseteq X\otimes g^jY$. By computing the Frobenius-Perron dimension of $X\otimes g^jY$ and its simple summands, we obtain
\begin{align*}
X\otimes g^jY=M\oplus \oplus_{k=1}^p g^kY,
\end{align*}  which also implies the following relations
\begin{align*}
M\otimes X=\oplus_{j=1}^pg^jY,~M\otimes g^jY=X.
\end{align*}
Similarly, we have  $M\otimes X=X\otimes X$ and $M\otimes g^jY=g^jY\otimes M$.  Particularly, $\text{Gr}(\M)$ is commutative.
However, it follows from  \cite[Theorem 3.23 (iii)]{NRWW} and \cite{GNN} that both $\C$ and $\D$ are self-dual modular fusion categories. Then  the algebra homomorphism
\begin{align*}
\text{Gr}(\Y(\M))\otimes_\Z\FQ\to\text{Gr}(\M)\otimes_\Z\FQ
 \end{align*}
 is surjective by \cite[Lemma 9.3.10]{EGNO}, so simple objects of $\M$ are self-dual, which is a contradiction. Hence, $\text{Gr}(\M)$ can't be commutative, so $g\otimes Y\cong Y\otimes g^{p-1}$, more generally,  $g^j\otimes Y\cong Y\otimes g^{p-j}$ for all $1\leq j\leq p-1$. Thus, for all $1\leq j,k\leq p-1$, we obtain
\begin{align*}
(g^j\otimes Y)\otimes (g^k\otimes Y)=g^j\otimes g^{p-k}\otimes Y\otimes Y=g^{j+p-k}\oplus X.
\end{align*}
In particular, $g^jY$ is self-dual for all $1\leq j\leq p$. Note that we still have
 \begin{align*}
\Hom_\M(X\otimes g^jY,g^kY)\cong\Hom_\M(X\otimes Y,g^kY)\cong\Hom_\M( X, g^{k+1}\oplus X)=\FC
 \end{align*}
for all $1\leq j,k\leq p$, then the fusion relations can be obtained in the same way.
\end{proof}

\begin{rema}
When $p=3$ and $q=7$, the fusion category $\M$ is exactly the fusion category $\mathcal{V}$ constructed by Ostrik in \cite[Proposition A.6.1]{CMS}.
\end{rema}

It is easy to see that one can construct a   fusion ring that is a $\Z_2$-extension of an arbitrary near-group fusion ring of type $(G,k|G|)$, where $G$ is abelian and $k$ is a nonnegative integer. However, for some non-cyclic abelian groups $G$, the corresponding  near-group fusion rings of type $(G,|G|)$ are  not categorifiable, one can take $G=\Z_2\times\Z_2\times\Z_2$ \cite[Proposition A.1]{Iz}\cite{Sch4}, for example, in these cases it is meaningless to consider the categorification of their extensions.

Hence, in the following definition, we only  list the corresponding fusion ring which contains a  near-group fusion ring of type $(\Z_n,n)$.
\begin{defi}
Let $R_0$ be a near-group fusion ring of type $(\Z_n,n)$ determined by the cyclic group $\Z_n=\langle g\rangle$ and relations
\begin{align*}
g^jg^l=g^{j+l},\quad g^jX=Xg^j=X,\quad XX=\sum_{g\in\Z_n}g+nX.
\end{align*}
Let $R\supseteq R_0$ be a fusion ring with $\Z_+$-basis $\{Y_j,g^j|1\leq j\leq n\}\cup\{M,X\}$ and the following fusion relations
\begin{align*}
MM&=\sum_{j=1}^ng^j,Y_jY_l=g^{j+n-l}+X, g^iY_j=Y_k=Y_jg^{n-i}(\text{where}~i+j\equiv k~\text{mod}~ n),\\
Y_jX&=XY_j=M+\sum_{l=1}^nY_l, MY_j=Y_jM=X, MX=XM=\sum_{j=1}^nY_j.
\end{align*}
\end{defi}
A direct computation shows \begin{align*}
\FPdim(X)=\frac{n+\sqrt{n^2+4n}}{2},~\FPdim(M)=\sqrt{n},~
\FPdim(Y_j)=\frac{\sqrt{n}+\sqrt{n+4}}{2},
\end{align*}
for all $1\leq j\leq n$. Then we obtain
\begin{align*}
\FPdim(R_0)=\frac{n^2+4n+n\sqrt{n^2+4n}}{2}, \FPdim(R)=n^2+4n+n\sqrt{n^2+4n}.
\end{align*}
 Hence \cite[Proposition 3.5.3]{EGNO} says that $R$ is a faithful $\Z_2$-extension of $R_0$.
Also notice that  $R$ contains a fusion ring (generated by $M$) of Frobenius-Perron dimension $2n$,  which is categorified as a Tambara-Yamagami fusion category $\mathcal{TY}(\Z_n,\tau,\mu)$.

In addition, we have the following proposition

 \begin{prop}
 When $n\leq 3$, $R$ is categorifiable. And there exists  a braided fusion category $\C$ such that $\text{Gr}(\C)=R$ if and only if  $n=1$.
  \end{prop}
  \begin{proof}
If $n=1$, then $\FPdim(M)=1$, and it is easy to see that \begin{align*}
   \text{Gr}(\C(\Z_2,\eta)\boxtimes\C(\ssl_2,3)_\text{ad})=R,
    \end{align*}
    where $\C(\ssl_2,3)_\text{ad}$ is the adjoint fusion subcategory of $\C(\ssl_2,3)$ \cite{BK,EGNO}. If $n=3$, then  $R$ is the Grothendieck ring of the fusion category $\mathcal{V}$
    \cite[Proposition A.6.1]{CMS}. When $n\geq3$, $R$ is non-commutative, obviously it can't be categorified as braided fusion category.

If $n=2$, then $\FPdim(R)=12+4\sqrt{3}$. We claim that it can be categorified by $\C(\ssl_2,10)_A$, where $A$ is a non-trivial connected \'{e}tale algebra  and $\FPdim(A)=3+\sqrt{3}$ by \cite[Theorem 6.5]{KiO}. Indeed, a  direct computation shows that the Frobenius-Perron dimensions of simple objects of $\C(\ssl_2,10)_A$ belong to $\{1,\sqrt{2},1+\sqrt{3},\sqrt{2+\sqrt{3}}\}$, and  $\sqrt{2+\sqrt{3}}=\frac{\sqrt{2}+\sqrt{6}}{2}$. Since $\C(\ssl_2,10)_A$ contains a unique simple object $X$ of Frobenius-Perron dimension $1+\sqrt{3}$ and two invertible objects $I,g$,  we obtain
   \begin{align*}
   g\otimes X=X=X\otimes g,\quad X\otimes X=I\oplus g\oplus 2X,
    \end{align*}
i.e., $X$ generates a near-group fusion category $\A$. Since $2\FPdim(\A)=\FPdim(\C(\ssl_2,10)_A)$, $\C(\ssl_2,10)_A$ admits a faithful $\Z_2$-grading with trivial component being $\A$ \cite[Proposition 3.5.3]{EGNO}, then the  rest fusion relations follow from the principal diagram \cite[Theorem 6.5]{KiO}.

However, when $n=2$, we claim that $R$  can't be categorified as braided fusion category even if it is commutative. On the contrary, assume that there is a braided fusion category $\B$   such that $\text{Gr}(\B)=R$. Since $\C$ always contains an Ising category $\I$ as fusion subcategory, which is modular by \cite[Corollary B.12]{DrGNO}, $\B\cong\I\boxtimes\D$ as braided fusion category \cite[Theorem 3.13]{DrGNO}, where $\D$ is a braided fusion subcategory of $\B$ such that   $\dim(\D)=3+\sqrt{3}$ by \cite[Theorem 3.14]{DrGNO}. So there exists a Galois conjugate of  $\D$ whose global dimension is $3-\sqrt{3}$, which contradicts the conclusion of
\cite[Theorem 1.1.2]{O3}.
  \end{proof}

  We end this section by proposing  the following question
\begin{ques}
Assume that there is a near-group fusion category $\A$ such that $\text{Gr}(\A)=R_0$. Is $R$ categorifiable when $n\geq4$?
\end{ques}
Indeed, $R$ is categorifiable when $n$ is odd and $\A$ exists and $\Y(\A)\cong\C(\Z_n,\eta)\boxtimes\C$ by the construction of $\M$ in Theorem \ref{categoryM}.

\section*{Acknowledgements}
The author  thanks V. Ostrik for comments on an early draft. The author is  supported by the National Natural Science Foundation of China (no.12101541), the Natural Science Foundation of Jiangsu Province (no.BK20210785), and the Natural Science Foundation of Jiangsu Higher Institutions of China (no.21KJB110006).

\bigskip\author{Zhiqiang Yu\\ \thanks{Email:\,zhiqyumath@yzu.edu.cn}\\{\small School  of Mathematical Science,  Yangzhou University, Yangzhou 225002, China}}

\end{document}